\documentclass[reqno, 11pt]{amsart}
\usepackage{ifthen}

\usepackage{setspace}
\usepackage[a4paper, left=3cm, right=3cm, top=4cm]{geometry}
\usepackage{amssymb, amsthm, epsf, epsfig}
\usepackage{amsmath}
\usepackage{verbatim}
\usepackage{caption}
\usepackage{subcaption}
\usepackage{dsfont}
\usepackage{wrapfig}
\usepackage[shortlabels]{enumitem}
\usepackage[dvipsnames]{xcolor}
\usepackage{natbib}
\usepackage{multirow}

\newtheorem{thm}{Theorem}
\newtheorem{thmn}{Theorem} 
\newtheorem{defn}{Definition}

\newtheorem{remark}{Remark}
\newtheorem{lemma}{Lemma}
\newtheorem{proposition}{Proposition}
\newtheorem{corollary}{Corollary}

\newcommand{\ZZ}{\mathds{Z}}

\newcommand{\mF}{\mathcal{F}}

\newcommand{\laplace}{\mathcal{L}}
\newcommand{\laplacex}{\laplace_{x=1}}

\author{Eva-Maria Hainzl}

\title{Formulas and asymptotics of hypergraph Catalan numbers}
\date{}

\begin{document}
\maketitle
\begin{abstract}
\emph{Tree walks} are a class of closed walks on a complete graph constrained to span trees. They appear in the computation of moments of the spectral measure of various random matrix models, most prominently in the spectral distribution of random graphs. In this work, we focus on a special subclass called \emph{$k$-tours}, which were introduced by Gunnells~\cite{gunn} after studying another random matrix model. They are enumerated by the so-called hypergraph Catalan numbers $ c_n^{(k)}$. Gunnells conjectured an asymptotic formula for $c_n^{(k)}$, which we confirm through an alternative approach to their enumeration. As it turns out, the asymptotic growth is governed by the number of $k$-tours on \emph{star-like} trees.
\end{abstract}

\section{Introduction}

\emph{Tree walks} are a class of closed walks on a complete graph constrained to span trees. To be precise, a \emph{tree walk} of length $2\ell$ is a closed sequence of vertices  
\[
    W = (v_1, v_2, \dots, v_{2\ell}, v_1)
\]  
such that the induced graph $T(W)$ on the visited vertices forms a tree. 
They were formally introduced in~\cite{pana_ha} with the motivation to understand the moments of the distribution of the eigenvalues of the adjacency matrix of the random graph $G(n,c/n)$ as $n$ tends to infinity. The enumeration of walks spanning trees however has a long history in spectral graph theory and was in particular used in~\cite{physics, inna_z} and various recent works~\cite{Bushygraphs, spec2, spec1,spec3,spec4}.

The main asymptotic growth of the moments of this distribution is governed by the asymptotic growth of tree walks and the edge probability $c/n$ amounts to a weight for the \emph{excess} of the tree walk, which is defined as  
\[
    \xi(W) = \ell - |E(T(W))|,
\]  
where $E(T(W))$ is the edge set of the induced tree $T(W)$. A key result of~\cite{pana_ha} established a connection between the generating function of tree walks and the classical \emph{Catalan generating function}  
\[
    C(z) = \frac{1 - \sqrt{1 - 4z}}{2z}
\]  
which is well-known to enumerate contour walks on trees with $n+1$ vertices~(see e.g.~\cite{Catalan}).
More precisely, if $W_{\xi}(z)$ is the generating function of tree walks with excess $\xi$, where $z$ marks the length of the walk, then
\[
    W_{\xi}(z) = C(z) \sum_{s=0}^{2\xi - 2} \frac{K_{\xi,s}(z C(z)^2)}{(1 - z C(z)^2)^{s+1}},
\]  
where $K_{\xi,s}(x)$ are polynomials with non-negative integer coefficients.\\

A special subclass of tree walks consists of those in which each edge is traversed \emph{exactly} the same number of times. These were introduced by Gunnells~\cite{gunn} as $a_T^{(k)}$-tours of a fixed tree $T$ and appeared in the study of a specific matrix model (see~\cite{matrixmod}). We will therefore refer to tree walks which cross each edge of the induced tree exactly $k$ times as \emph{$k$-tours}. 

These $k$-tours on trees with $n+1$ vertices are enumerated by the \emph{hypergraph Catalan numbers} $c_n^{(k)}$ (see \cite[Definition~2.5]{gunn}).  
They are a generalization of the classical Catalan numbers, extending their combinatorial interpretation to special enumerated trees, lattice walks, triangulations, polygon gluings and more. Gunnells found a formal way of computing $c_n^{(k)}$ and conjectured an asymptotic growth formula for them. In this paper, we confirm Gunnells’ conjecture by employing a different perspective based on the enumeration of general tree walks. Our main result is thus the following theorem.

\begin{thm}
    Let $k\geq 1$ and $c_n^{(k)}$ be the number of $k$-tours on trees with $n+1$ vertices. Then
    \begin{align*}
        c_n^{(1)} = \frac{1}{n+1}\binom{2n}{n}\sim \sqrt{\frac{1}{\pi n^3}} \,4^n,\qquad
        c_n^{(2)} \sim 2\,\sqrt{\frac{e^{3}}{\pi n}} \, 2^{n}n!
    \end{align*}
    and for $k\geq 3$, 
    \begin{align*}
        c_n^{(k)} \sim 2\,\sqrt{\frac{k}{(2\pi n)^{k-1}}}\left(\frac{k^k}{k!}\right)^n(n!)^{k-1},
    \end{align*}
    as $n\rightarrow \infty$.
\end{thm}

In the following section, we focus on counting $k$-tours, providing an explicit formula (Theorem~\ref{lem:closed_form}) by using elementary counting arguments and via an alternative approach through generating functions. Subsequently, we conduct an asymptotic analysis of this formula in Section~3 highlighting that the main contribution to the asymptotic growth of $k$-tours arises from star-like trees~(Proposition~\ref{prop:star}). 
\section{Counting $k$-tours}

While the exact enumeration of general tree walks is still challenging, $k$-tours are a relatively simple subclass since they cross each edge the same number of times and their length and excess therefore solely depends on the size of the tree along which they walk. Indeed, it is possible to count $k$-tours on trees with $n+1$ vertices with elementary arguments as the first proof of the following theorem shows. Subsequently, we provide a second proof of the formula which relies on the same steps which were conducted in the recursive decomposition of tree walks in~\cite{pana_ha}.

\begin{thm}\label{lem:closed_form}
    Let $k\geq 1$ and $c_n^{(k)}$ be the number of $k$-tours on trees with $n+1$ vertices. Then
    \[
        c_n^{(k)} = \sum_{\ell=1}^{n} 
        \sum_{\substack{(n_0,n_1,\dots)\in T_n(\ell)}}
        \frac{\ell}{n}\binom{n}{n_0,n_1,n_2,\dots}
        \frac{1}{\ell!}\binom{\ell k}{k,k,\dots,k} \prod_{i\geq 1}\left(\frac{1}{(i+1)!}\binom{(i+1)k}{k,k,\dots,k,k}\right)^{n_i}
    \]
    where $T_n(\ell) = \left\{(n_0,n_1,\dots, n_{n-1}) \in \left(\ZZ_{\geq 0}\right)^{n} \mid \sum_{i=0}^{n-1} n_i = n, \sum_{i=0}^{n-1} i n_i = n-\ell\right\}$ is the set of degree sequences of trees with root degree $\ell$ and $\frac{\ell}{n}\binom{n}{n_0,n_1,n_2,\dots}$ the number of trees with the corresponding degree sequence.
\end{thm}

\begin{remark}
    The formula in Theorem~\ref{lem:closed_form} holds for all $k\geq 1$. If we set $k=1$, we obtain
        \begin{align*}
            c^{(1)}_n &= \sum_{\ell \geq 1} \sum_{\substack{(n_0,n_1,\dots) \in T_n(\ell)}} \frac{\ell}{n}\binom{n}{n_0,n_1,n_2,\dots}
        \end{align*}
    which, by Theorem 6.4. in~\cite{gess} equals the sum over all plane forests with $\ell$ trees and $n$ vertices - or equivalently the sum over all plane trees with $n+1$ vertices with root degree $\ell$.
    
    The number of plane trees on $n+1$ vertices in turn equals the $n$-th Catalan number $c_n = c_n^{(1)}$.
\end{remark}

\begin{proof}[Proof of Theorem~\ref{lem:closed_form}]
    Each $k$-tour on a rooted tree defines a canonical embedding of the tree in the plane, given by the order in which the tour discovers the vertices. So, we may think about $k$-tours as walks on rooted plane trees which start and end at the root, cross each edge exactly $2k$ times and the order in which outgoing edges are crossed the first time at a vertex $v$, is the order from left to right of the edges.
    Thus, if $k=1$, this condition uniquely determines the walk and therefore $c_n^{(1)} = c_n$, where $c_n$ is the $n$-th Catalan number which counts the number of plane trees with $n+1$ vertices.
    
    If $k>1$, we notice that the order in which the walk leaves each vertex fully determines the walk. 
    \begin{figure}
    \begin{subfigure}{0.32\textwidth}
        \centering
        \includegraphics[width=1\linewidth, page=1]{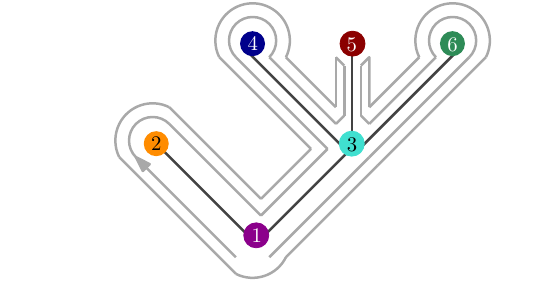}
        \caption{The walk $W$}
        \label{fig:greytour}
    \end{subfigure}
    \begin{subfigure}{0.32\textwidth}
        \centering
        \includegraphics[width=1\linewidth, page=2]{figures/ktours.pdf}
        \caption{Outgoing steps of $W$}
        \label{fig:coltour}
    \end{subfigure}
    \begin{subfigure}{0.32\textwidth}
        \centering
        \includegraphics[width=1\linewidth, page=3]{figures/ktours.pdf}
        \caption{$W$ as sequences}
        \label{fig:permtour}
    \end{subfigure}
    \caption{The tree walk $W = (1,2,1,3,4,3,5,3,4,3,1,2,1,3,6,3,5,6,3,1)$ separated into outgoing steps at vertices and into sequence of vertices} \label{fig:ktours}
    \end{figure}
    Figure~\ref{fig:ktours} illustrates how one can separate a tree walk into sequences of outgoing steps at each step. One simply records at each vertex $v$ the sequence $S_v$ of neighbors in which they are ($k$-times) visited starting from $v$. 
    For an example, see Figure~\ref{fig:ktours} where we record $S_1 = (2,3,2,3)$ at the root $1$ for the walk $W = (1,2,1,3,4,3,5,3,4,3,1,2,1,3,6,3,5,6,3,1)$.

    The resulting sequence for a vertex $v$ is a sequence $S_v = (s_{v}(1),s_{v}(2),\dots,s_v{(kd_v)})$ which contains exactly $k$ occurrences of each of its $d_v$ neighbors, the first occurrences of its children correspond to their order in the tree and the last vertex in the sequence is its parent. \\
    
    Given a rooted plane tree (w.l.o.g. with root vertex $1$), where each vertex $v$ is assigned a sequence with the above properties, we can reconstruct $W$ by going through the sequences step-by-step, as the following algorithm shows.\\
    
    \noindent 1\qquad Initalize $W=\emptyset$, $v=1$\\
    2\qquad WHILE $v\ne 0$ DO:\\
    3\qquad \qquad Append $v$ to $W$\\
    4\qquad \qquad IF $s_v(1) \ne \emptyset$ THEN:\\
    5\qquad \qquad \qquad Delete $s_v(1)$ from $S_v$\\
    6\qquad \qquad \qquad Set $v=s_v(1)$\\
    7\qquad\qquad ELSE:\\
    8\qquad \qquad \qquad Set $v=0$\\
    9\qquad RETURN $W$\\

    This reconstruction is always possible and uniquely reproduces the walk.
    In order to enumerate all $k$-tours, we therefore fix a plane tree $T$ and locally count the number of ways how the $k$-tour could leave each vertex. 

    So, let $d_v$ be the degree of $v\in V$. The walk leaves $v$ exactly $kd_v$ times and the number of ways to distribute these (ordered) steps on $d_v$ edges is 
    \[\binom{kd_v}{k,k,\dots,k}.\]
    However, these edges appear in a specific order in the underlying tree. Since there are $d_v!$ ways to permute the order of the edges, the number of ways to distribute the outgoing steps at $v$ is therefore
    \[
        \frac{1}{d_v!}\binom{kd_v}{k,k,\dots,k}.
    \]

    Hence, if the root of $T$ has degree $\ell$ and there $n_i$ further vertices with degree $i+1$ for each $i\in \{1,2,\dots,n-1\}$, then the number of $k$-tours on $T$ is
    \[
        \frac{1}{\ell!}\binom{\ell k}{k,k,\dots,k}\prod_{i\geq0}\left(\frac{1}{(i+1)!}\binom{(i+1)k}{k,\dots,k}\right)^{n_i}.
    \]
    Further, by~\cite[Theorem 6.4]{gess} there are 
    \[
        \frac{\ell}{n}\binom{n}{n_0,n_1,\dots}
    \]
    rooted trees with root degree $\ell$ and $n_i$ interior vertices with $i$ children respectively. Summing over all possible degree sequences results in the formula in Theorem~\ref{lem:closed_form}.
\end{proof}

An alternative proof of Theorem~\ref{lem:closed_form} can be achieved via an approach by generating functions as the following lemma shows.

\begin{lemma}\label{lem:hyper}
    Let $c^{(k)}_{j,n}$ denote the number of $k$-tours on a tree with $n+1$ vertices and a root with degree $j$. Let further
    \[ 
        C_{k}(x,z) = \sum_{j,n\geq 0} c^{(k)}_{j,n} \frac{x^{kj}}{(kj)!}z^{n+1}
    \]
    denote their generating function where $z$ marks the size of the tree and $x$ marks the number of departing steps from the root. Then
    \[
        C_{k}(x,z) = z\exp\left(\frac{x^k}{k!} \laplace_{t=1}\left(\frac{t^{k-1}}{(k-1)!}C_{k}(t,z)\right)\right)
    \]
    where $\laplace_{t=1} \left( A(t) \right) = \sum_{k \geq 0} k!\, [t^k] A(t)$.
\end{lemma}

The proof is analogous to Lemma 9 in~\cite{pana_ha} but we include an adjusted version for completeness.
\begin{proof}
    Let $\mF$ denote the family of $k$-tours on trees with $n+1$ vertices and with root degree $1$.
    Let $f_{k,1,m,n}$ denote the number of such $k$ tours on trees with root degree $1$ and $m+1$ is the degree of the only child of the root. Hence, these tours leave the root $k$ times and the only child of the root $k(m+1)$ times.
    We associate to the family $\mF$ the generating function
    \[
        F_k(t,x,z) =
        \sum_{m,n \geq 0}
        f_{k,1,m,n}
        \frac{t^{k(m+1)-1}}{(k(m+1)-1)!}
        \frac{x^k}{k!}
        z^{n},
    \]
    where $z$ counts all vertices except for the root, $x$ the departing steps from the root and $t$ the  departing steps from $c$ except for the last step.
    Consider a walk $W \in \mF$,
    let $r$ denote its root,
    $c$ the only child of $r$,
    and $W'$ the walk of root $c$
    obtained from $W$ by removing $r$.
    Assume the number of steps leaving $c$ is $(m+1)k$.
    We construct two disjoint sets $A$ and $B$
    whose union is $\{1, 2, \ldots, (m+1)k-1\}$.
    For all $i \in \{1,2,\ldots,(m+1)k-1\}$,
    if the $i$-th step leaving $c$
    goes to the root $r$, then $i \in A$.
    Otherwise, $i \in B$.
    We did not include the $(m+1)k$-th step leaving $c$, as we know it must always reach $r$. Observe that knowing $A$, $B$ and $W'$ is enough to reconstruct $W$. Further, $A$ has of course size $k-1$ since $W$ is a $k$-tour and there are $k$ steps leaving its root $r$.
    Because of the choice to make $t$ an exponential variable in the generating function $F_k(t,x,z)$, this bijection implies
    \[
        F_k(t,x,z) =      
        \frac{t^{k-1}x^k}{(k-1)!k!}C_k(t,z).
    \]
    In order to forget the variable $t$ in the generating function $F_k(t,x,z)$, we use the Laplace operator
    \[
        \sum_{m,n \geq 0}
        f_{k,1,m,n}
        \frac{x^{k}}{k!}
        z^{n}
        =
        \laplace_{t=1}(F_k(t,x,z)) = \frac{x^k}{k!}\laplace_{t=1}\left(\frac{t^{k-1}}{(k-1)!}C_k(t,z)\right).
    \]
    Finally, a $k$-tour is a root
    and a (possibly empty) set
    of elements from $\mF$, so
    \[
        C_k(x,z) = z \exp \left( \laplace_{t=1}(F_k(t,x,z)) \right).
    \]
\end{proof}
\begin{figure}
        \centering
        \includegraphics[width=0.3\textwidth, page=7]{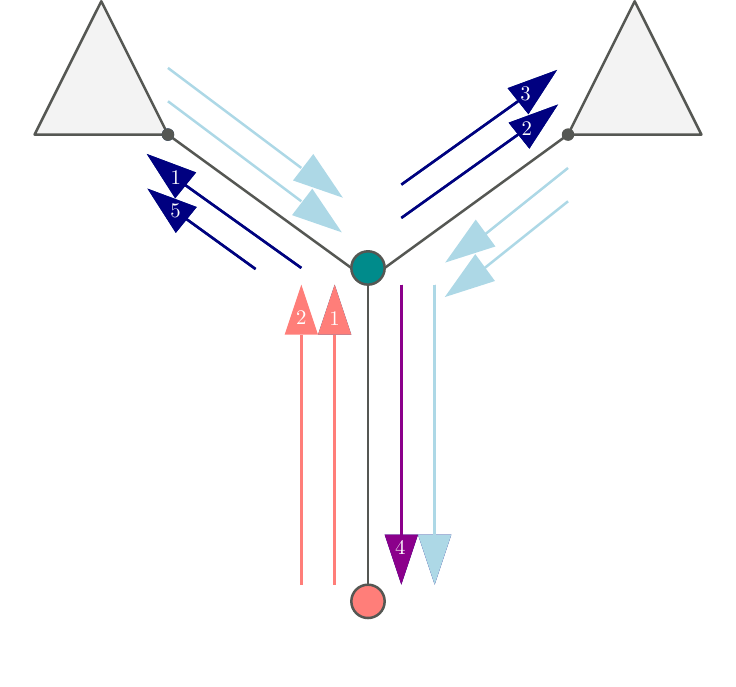}
        \caption{Decomposition of a 2-tour in $\mF$. The root is colored red and its only child $c$ green. Steps in set $A$ are purple and steps in set $B$ are dark blue.}
        \label{fig:decomp_tree}
\end{figure}
 Now given Lemma~\ref{lem:hyper}, we can prove Theorem~\ref{lem:closed_form} using Lagrange inversion.
 
\begin{proof}[Alternative proof of Theorem~\ref{lem:closed_form}]
    According to Lemma \ref{lem:hyper}, it holds that
    \[
        C_{k}(x,z) = z\exp\left(\frac{x^k}{k!} \laplace_{t=1}\left(\frac{t^{k-1}}{(k-1)!}C_{k}(t,z)\right)\right)
    \]
    and therefore it holds for $A(z) = \laplacex \left(\frac{x^{k-1}}{(k-1)!}C_{k}(x,z)\right)$
    \begin{align*}
        A(z) &= z\laplacex\left(\frac{x^{k-1}}{(k-1)!}\exp\left(\frac{x^k}{k!} A(z)\right)\right)\\
        &= z\sum_{i\geq 0} \frac{(ik+k-1)!}{(k-1)!(k!)^ii!}A(z)^i\\
        &= z \sum_{i\geq 0} \binom{(i+1)k-1}{k,k,\dots,k,k-1}\frac{A(z)^i}{i!}\\
        &= z \sum_{i\geq 0} \binom{(i+1)k}{k,k,\dots,k,k}\frac{A(z)^i}{(i+1)!}.
    \end{align*}
    If we apply the $\laplacex$-operator to $C_k(x,z)$ analogously, we obtain
    \[  
        C_k(z) = z\laplacex\left(\exp\left(\frac{x^k}{k!} A(z)\right)\right) = z\sum_{i\geq 0} \binom{ik}{k,k,k,\dots,k}\frac{A(z)^i}{i!}
    \]
    Consequently, we can apply the Lagrange inversion theorem for
    \[
        A(z) = z\phi(A(z)),\quad C_k(z) = zH(A(z))
    \]
    where $\phi(u) = \sum_{i\geq 0} \binom{(i+1)k}{k,k,\dots,k,k}\frac{u^i}{(i+1)!}$ and $H(u) = \sum_{i\geq 0} \binom{ik}{k,k,k,\dots,k}\frac{u^i}{i!}$. Thus, we obtain 
    \begin{align*}
        c^{(k)}_n &= [z^{n+1}]C_k(z) = \frac{1}{n}[u^{n-1}]H'(u)\phi(u)^{n}\\
        &= \frac{1}{n}\sum_{\ell\geq0}\sum_{\substack{n_0+n_1+\dots = n\\\ell+n_1+2n_2+...=n-1}}  \binom{n}{n_0,n_1,n_2,\dots}\frac{1}{\ell!}\binom{(\ell+1) k}{k,k,\dots,k}\prod_{i\geq 1}\left(\frac{1}{(i+1)!}\binom{(i+1)k}{k,k,\dots,k,k}\right)^{n_i}
    \end{align*}
    which is equivalent to the formula in Theorem~\ref{lem:closed_form} up to an index shift on $\ell$.
\end{proof}

\section{Asympotics of the hypergraph Catalan numbers}

Given the closed formula in Theorem~\ref{lem:closed_form}, we may prove the asymptotic formulas which we already stated in the introduction. 

\begin{thmn}
    Let $k\geq 1$ and $c_n^{(k)}$ be the number of $k$-tours on trees with $n+1$ vertices. Then
    \begin{align*}
        c_n^{(1)} = \frac{1}{\sqrt{\pi n^3}}4^n,\qquad
        c_n^{(2)} = 2\,\sqrt{\frac{e^{3}}{\pi n}} \,2^nn!
    \end{align*}
    and
    \begin{align*}
        c_n^{(k)} = 2\,\sqrt{\frac{k}{(2\pi n)^{k-1}}}\left(\frac{k^k}{k!}\right)^n(n!)^{k-1}, \quad k\geq3
    \end{align*}
    as $n\rightarrow \infty$.
\end{thmn}

Note that Gunnell conjectured asymptotics of the form
\[
    c_n^{(k)} = 
    \begin{cases}
        \frac{2\binom{2}{2}\binom{4}{2}\cdots\binom{k-1}{2}}{k^{(2k-3)/2}(\pi n)^{(k-1)/2}}\left(\frac{k^k}{k!}\right)^{n+1}(n!)^{k-1} &\mbox{if k is odd}\\
        \frac{\sqrt{2}\binom{3}{2}\binom{5}{2}\cdots\binom{k-1}{2}}{k^{(2k-3)/2}(\pi n)^{(k-1)/2}}\left(\frac{k^k}{k!}\right)^{n+1}(n!)^{k-1} &\mbox{if k is even and k}>2
    \end{cases}
\]
However, it is quite straightforward to check that these expressions coincide with our formula. For $k$ odd\footnote{In the case where $k$ is even, we assume that a typo must have happened in the conjecture. E.g. note that the product of binomial coefficients should probably end at $\binom{k}{2}$ or $\binom{k-2}{2}$.}, we can simplify 
\begin{align*}
    \frac{2\binom{2}{2}\binom{4}{2}\cdots\binom{k-1}{2}}{k^{(2k-3)/2}(\pi n)^{(k-1)/2}}\left(\frac{k^k}{k!}\right)^{n+1}(n!)^{k-1}
    &= \frac{2(k-1)!}{2^{(k-1)/2}k^{k-3/2}(\pi n)^{(k-1)/2}}\left(\frac{k^{k-1}}{(k-1)!}\right)\left(\frac{k^k}{k!}\right)^{n}(n!)^{k-1} \\
    &= \frac{2k^{1/2}}{2^{(k-1)/2}(\pi n)^{(k-1)/2}}\left(\frac{k^k}{k!}\right)^{n}(n!)^{k-1} \\
    &= 2\sqrt{\frac{k}{(2\pi n)^{k-1}}}\left(\frac{k^k}{k!}\right)^{n}(n!)^{k-1}.
\end{align*}
Indeed, in this section we will show that the asymptotics are governed by $k$-tours on trees that look like stars. 

\begin{defn}[Stars and star-like trees]
    Let $n\geq 1, m \in \{0,1,2\dots,n-3\}$ and $T$ be a (rooted) tree with $n+1$ vertices. If $T$ has a single vertex with degree $n-m$, $m$ vertices of degree $2$ and $n-m$ leaves, then we call $T$ \emph{star-like}. If further $m=0$, that is $T$ is a tree with one vertex of degree $n$ and $n$ leaves, then $T$ is called a \emph{star}.
 \end{defn}

The following proposition shows that the number of $k$-tours on stars (and star-like trees) is the main driver in the asymptotic growth of $c_n^{(k)}$.

\begin{figure}
        \centering
        \begin{subfigure}{0.2\textwidth}
            \includegraphics[width=\linewidth, page=1]{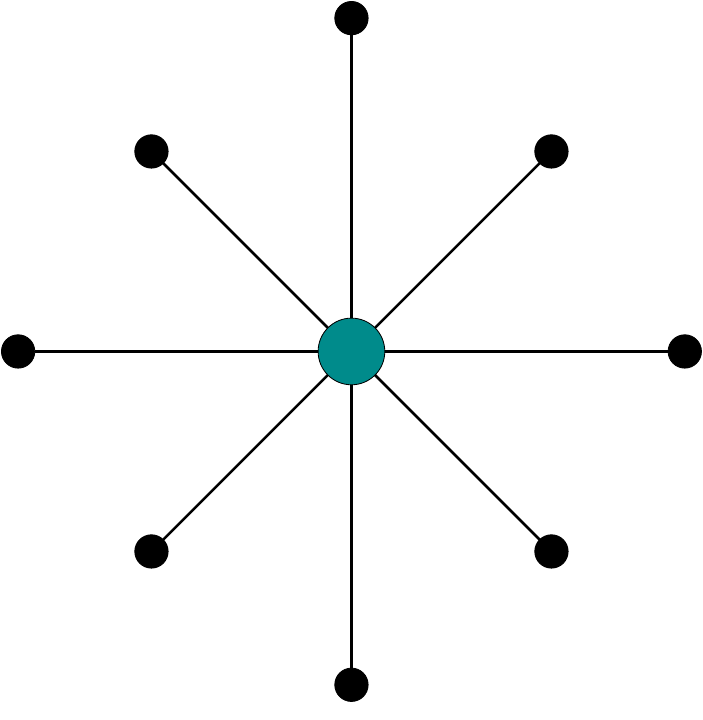}
        \end{subfigure}
        \hspace{10mm}
        \begin{subfigure}{0.2\textwidth}
            \includegraphics[width=\linewidth, page=2]{figures/star.pdf}
        \end{subfigure}
        \caption{Stars: the only rooted trees contributing to the asymptotic growth of $c_n^{(k)}$ for $k>2$. (The root is marked green)}
        \label{fig:startree}
    \end{figure}

\begin{proposition}\label{prop:star}
    Let $k\geq 2, m = o(\sqrt{n})$ and let $s_k(n,m)$ be the number of $k$-tours on a star-like trees with maximum vertex degree $n-m$. Then for $k\geq 2$, we have
    \[
        s_k(n,0) \sim  2\sqrt{\frac{k}{(2\pi n)^{k-1}}}\left(\frac{k^k}{k!}\right)^n(n!)^{k-1},
        \qquad 
        s_k(n,m) \sim  \frac{1}{m!}\left(\frac{(2k)!}{2k^kk!}\right)^m  \frac{s_k(n,0)}{n^{(k-2)m}}
    \]
    as $n$ tends to infinity. 
\end{proposition}

This suggests, that for $k\geq 3$ there is a single tree, namely a tree with one vertex of degree $n$, that accounts for the asymptotics of $c_n^{(k)}$ (see Figure~\ref{fig:startree}). 

For $k=2$, we sum up the number of $k$-tours on star-like trees and obtain
    \[
        \sum_{m=0}^{o(\sqrt{n})} s_2(n,m) \sim s_2(n,0)\left(\sum_{c=0}^{o(\sqrt{n})} \frac{1}{m!}\left(\frac{3}{2}\right)^m\right) \sim 2\,\sqrt{\frac{e^3}{\pi n}}\, 2^{n}
        n!
    \]
   which corresponds to the asymptotics of $c_n^{(2)}$.\\

   In both cases, our aim is to show in the following that the rest of the $k$-tours are asymptotically negligible. However, let us first prove the proposition above. The proof is a straightforward computation based on the formula in Proposition~\ref{lem:closed_form}. 

\begin{proof}[Proof of Proposition~\ref{prop:star}]
    We have to distinguish three cases. 
    \begin{enumerate}
        \item If the root has degree $\ell = n-m$, then $n_1 = n-(n-m) = m$ and $n_0 = n-n_1 = n-m$. 
        \item If there is an internal vertex with degree $n-m$ (and therefore outdegree $n-m-1$) and the root has degree $\ell = 1$, then $n_1 = n-1-(n-m-1) = m$, while $n_0 = n-n_1-1 = n-m-1$.
        \item If there is an internal vertex with degree $n-m$ (and therefore outdegree $n-m-1$) and the root has degree $\ell = 2$, then $n_1 = n-2-(n-m-1) = m-1$ and $n_0 = n-n_1-1 = n-m$.
    \end{enumerate}
    Hence, by our computations in the proof of Lemma~\ref{lem:closed_form},
    \begin{align*}
        s_k(n,m) &= \frac{n-m}{n}\binom{n}{n-m,m}\left(\frac{1}{(n-m)!}\binom{k(n-m)}{k,\dots,k}\right)\left(\frac{1}{2!}\binom{2k}{k,k}\right)^m \\
        &\quad+ \frac{1}{n}\binom{n}{n-m-1,m,1}\left(\frac{1}{2!}\binom{2k}{k,k}\right)^m\left(\frac{1}{(n-m)!}\binom{k(n-m)}{k,\dots,k}\right)\\
        &\quad+ \frac{2}{n}\binom{n}{n-m,m-1,1}\left(\frac{1}{2!}\binom{2k}{k,k}\right)\left(\frac{1}{2!}\binom{2k}{k,k}\right)^{m-1}\left(\frac{1}{(n-m)!}\binom{k(n-m)}{k,\dots,k}\right)
    \end{align*}
    and therefore,
    \begin{align*}
        s_k(n,m)
        &=\frac{2}{m!}\left(\frac{1}{2}\binom{2k}{k,k}\right)^m\frac{(n-1)!}{(n-m-1)!}\left(\frac{1}{(n-m)!}\binom{k(n-m)}{k,\dots,k}\right)\\
        &\quad+\frac{2}{(m-1)!}\left(\frac{1}{2}\binom{2k}{k,k}\right)^m\frac{(n-1)!}{(n-m)!}\left(\frac{1}{(n-m)!}\binom{k(n-m)}{k,\dots,k}\right).
    \end{align*}
    
    For $m=0$, the above simplifies to
    \begin{align}\label{eq:a_k(n)}
        \frac{s_k(n,0)}{(n!)^{k-1}} 
        &= \frac{2}{(n!)^{k}}\binom{kn}{k,\dots,k}\left(1+O\left(\frac{1}{n}\right)\right)
        \sim 2\,\sqrt{\frac{k}{(2\pi n)^{k-1}}}\left(\frac{k^k}{k!}\right)^{n}
    \end{align}
    If $m\geq 1$ and we compare this number to $s_k(n,0)$, we obtain
    \begin{align*}
        \frac{s_k(n,m)}{s_k(n,0)} &=\frac{1}{m!}\left(\frac{k!}{2}\binom{2k}{k,k}\right)^m\frac{(n-1)!}{(n-m-1)!}\frac{n!}{(n-m)!}\frac{(k(n-m))!}{(kn)!}\left(1+\frac{m}{n-m}\right)
    \end{align*}
    Let us assume that $m=o(\sqrt{n})$, such that we can apply Stirling's approximation to all six factorials at the end of this expression. That is,
    \begin{align*}
        \frac{s_k(n,m)}{s_k(n,0)} &\sim \frac{1}{m!}\left(\frac{(2k)!}{2k!}\right)^m\sqrt{\frac{2\pi(n-1)\,2\pi n\,2\pi k(n-m)}{2\pi(n-m-1)\, 2\pi(n-m)\,2\pi kn}}\left(\frac{n-1}{e}\right)^{n-1}\left(\frac{n}{e}\right)^{n}\\
        &\qquad \cdot \left(\frac{k(n-m)}{e}\right)^{k(n-m)}\left(\frac{e}{n-m-1}\right)^{n-m-1}\left(\frac{e}{n-m}\right)^{n-m}\left(\frac{e}{kn}\right)^{kn}\\
        &\sim \frac{1}{m!}\left(\frac{(2k)!}{2k^kk!}\right)^m\sqrt{\frac{n-1}{n-m-1}}\frac{n-m}{n-1}\\
        &\qquad \cdot \left(\frac{n-1}{n}\right)^n\left(\frac{n-m}{n-m-1}\right)^{n-m-1}\left(\frac{n-m}{n}\right)^{(k-2)(n-m)}\left(\frac{e}{n}\right)^{(k-2)m}
    \end{align*}
    Now we can use the fact that $(1+\frac{m}{n})^n \sim e^{m}$ for $m=o(\sqrt{n})$ and obtain the desired asymptotic expression
    \begin{align*}
        \frac{s_k(n,m)}{s_k(n,0)} \sim \frac{1}{m!}\left(\frac{(2k)!}{2k^kk!}\right)^m n^{-(k-2)m}
    \end{align*}
    as stated in the Lemma.
\end{proof}

\begin{figure}
        \centering
        \begin{subfigure}{0.2\textwidth}
            \includegraphics[width=\linewidth, page=1]{figures/star.pdf}
            \vspace{0mm}
            \caption{Star(-like) tree}
        \end{subfigure}
        \hspace{6mm}
        \begin{subfigure}{0.35\textwidth}
            \includegraphics[width=\linewidth]{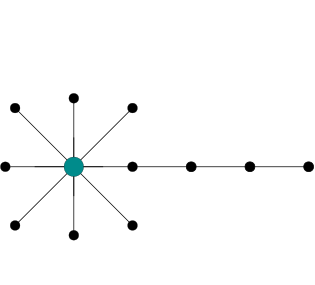}
            \caption{Star-like tree}
        \end{subfigure}
        \hspace{6mm}
        \begin{subfigure}{0.25\textwidth}
        \includegraphics[width=\linewidth, page=2]{figures/starlike.pdf}
        \vspace{0mm}
        \caption{Star-like tree}
        \end{subfigure}
        \caption{Star-like trees: rooted trees contributing to the asymptotic growth of $c_n^{(2)}$. (The root is marked green)}
        \label{fig:starlikeT}
    \end{figure}
We observe that in the proof above the number of $k$-tours on a trees with one vertex of degree $n-m$ and $m$ internal vertices with degree $2$ is of course the same for each tree. That is, the number of $k$-tours is consistently 
\[
    \left(\frac{1}{2}\binom{2k}{k,k}\right)^m\frac{1}{(n-m)!}\binom{k(n-m)}{k,\dots,k}.
\]
However, when the root is of degree $2$, there are less rooted trees such that only the first two cases which we considered contribute asymptotically.
We generalize this observation such that we may only estimate the number of $k$-tours on trees with root degree $1$ in our final computations.

\begin{lemma}[Rerooting lemma]\label{lem:reroot}
    Let $n\in \mathbb{N}$, $m\in \{1,2,  \dots, n-1\}$ and $n_1,n_2,\dots n_{n-1}$ such that $\sum_{i=1}^{n-1} n_i = m-1$ and $\sum_{i=1}^{n-1} in_i = n-1$. Then, 
    \[
        \sum_{\ell =2}^{n-(m-1)} \frac{\ell}{n}\binom{n}{n-m+1,n_1,\dots, n_{\ell-1}-1,\dots} \leq \frac{m}{n}\binom{n}{n-m,n_1,\dots, n_{\ell-1},\dots}.
    \]
\end{lemma}

\begin{proof} Intujtively, the statement is clear, because each tree with $n+1-m$ leaves, which is rooted at a leaf can obviously be rerooted at $m$ internal vertices. Let us compute the bound nevertheless.
    \begin{align*}
        \frac{\ell}{n}\binom{n}{n-m+1,n_1,\dots, n_{\ell-1}-1,\dots} 
        &= \frac{\ell n_{\ell-1}}{n-m+1}\frac{1}{n}\binom{n}{n-m,n_1,\dots, n_{\ell-1},\dots},
    \end{align*}
    Summing over $\ell$ then results in the desired identity since
    \begin{align*}
        \sum_{\ell=2}^{n-m+1} \frac{\ell n_{\ell-1}}{n-m+1} &= \frac{1}{n-m+1}\sum_{\ell=2}^{n-m+1} \left((\ell-1) n_{\ell-1}+n_{\ell-1}\right) \\
        &= \frac{n-1}{n-m+1}+\frac{m}{n-m+1} \\
        &= 1 + \frac{2(m-1)}{n-m+1}
    \end{align*}
    and it is straightforward to check that
    \[
        \frac{2(m-1)}{n-(m-1)} \leq (m-1)
    \]
    under the assumption that $1\leq m\leq n-1$.
\end{proof}

\begin{corollary}\label{cor:reroot}
    As a direct consequence, we obtain
    \begin{align*}
        &\sum_{\ell=1}^{n} 
        \sum_{\substack{(n_0,n_1,\dots)\in T_n(\ell)}}
        \frac{\ell}{n}\binom{n}{n_0,n_1,n_2,\dots}
        \frac{1}{\ell!}\binom{\ell k}{k,k,\dots,k} \prod_{i\geq 1}\left(\frac{1}{(i+1)!}\binom{(i+1)k}{k,k,\dots,k,k}\right)^{n_i}\\
        &\leq \sum_{\substack{(n_0,n_1,\dots)\in T_n(1)}}\frac{n-n_0+2}{n}\binom{n}{n_0,n_1,n_2,\dots} \prod_{i\geq 1}\left(\frac{1}{(i+1)!}\binom{(i+1)k}{k,k,\dots,k,k}\right)^{n_i}
    \end{align*}
\end{corollary}

\begin{proof}
    For $(n_0,n_1,\dots)\in T_n(\ell)$, we denote the number of vertices in the tree of degree $\ell$ by $\tilde{n}_\ell = n_{\ell-1}+1$ and the number of vertices of degree $i\ne \ell$ by $\tilde{n}_i = n_{i-1}$. Now, note that 
    \begin{align*}
        &\sum_{\ell=1}^{n} 
        \sum_{\substack{(n_0,n_1,\dots)\in T_n(\ell)}}
        \frac{\ell}{n}\binom{n}{n_0,n_1,n_2,\dots}
        \frac{1}{\ell!}\binom{\ell k}{k,k,\dots,k} \prod_{i\geq 1}\left(\frac{1}{(i+1)!}\binom{(i+1)k}{k,k,\dots,k,k}\right)^{n_i}\\
        &= \sum_{\ell=1}^{n} 
        \sum_{\substack{(n_0,n_1,\dots)\in T_n(\ell)}}
        \frac{\ell}{n}\binom{n}{n_0,n_1,n_2,\dots}
        \prod_{\substack{i\geq 1}}\left(\frac{1}{i!}\binom{ik}{k,k,\dots,k,k}\right)^{\tilde{n}_i}
    \end{align*}
    If $\ell>1$, rerooting the tree to the left-most leaf gives a tree with $(n_0-1,n_1,\dots,n_{\ell-1}+1,\dots)\in T_n(1)$. The $\tilde{n}_i, i\geq 1$ are of course not affected by this rerooting. Now, the Rerooting Lemma~\ref{lem:reroot} tells us that the sum over all rooted trees can be estimated by
    \[
        \sum_{\substack{(\tilde{n}_1,\tilde{n}_2,\dots)\in T_n(1)}}
        \frac{m+1}{n}\binom{n}{n_0,n_1,n_2,\dots}
        \prod_{\substack{i\geq 1}}\left(\frac{1}{i!}\binom{ik}{k,k,\dots,k,k}\right)^{\tilde{n}_i}
    \]
    where $m-1 = n-n_0$.
\end{proof}

Finally, we are ready to tackle the asymptotic analysis of the formula in Theorem~\ref{lem:closed_form}.

\begin{proof}[Proof of Theorem 1]
    By Lemma~\ref{lem:closed_form}, we want to analyze the sum
    \begin{equation}\label{eq:bk,k>2}
        c_n^{(k)} = \sum_{\ell=1}^{n} 
        \sum_{\substack{(n_0,n_1,\dots)\in T_n(\ell)}}
        \frac{\ell}{n}\binom{n}{n_0,n_1,n_2,\dots}
        \frac{1}{\ell!}\binom{k \ell}{k,k,\dots,k} \prod_{i\geq 1}\left(\frac{1}{(i+1)!}\binom{(i+1)k}{k,k,\dots,k,k}\right)^{n_i}.
    \end{equation}
    To that end, we separate the sum into two parts depending on the highest degree in the tree. Let $M(\ell, n_1,\dots) := \max\left\{\{\ell\} \cup \{i+1 \mid n_1\geq 1\}\right\}$ denote the highest degree in the tree with the respective degree sequence.
    \begin{align*}
        c_n^{(k)} 
        &= \sum_{\ell=1}^{n} 
        \sum_{\substack{(n_0,n_1,\dots)\in T_n(\ell)\\
        M(\ell,n_1,\dots) < n-\lceil n^{1/3}\rceil}}
        \frac{\ell}{n}\binom{n}{n_0,n_1,n_2,\dots}
        \frac{1}{\ell!}\binom{k \ell}{k,k,\dots,k} \prod_{i\geq 1}\left(\frac{1}{(i+1)!}\binom{(i+1)k}{k,k,\dots,k,k}\right)^{n_i} \\
        &+ \sum_{\ell=1}^{n} 
        \sum_{\substack{(n_0,n_1,\dots)\in T_n(\ell)\\
        \\
        M(\ell,n_1,\dots) \geq n-\lceil n^{1/3}\rceil}}
        \frac{\ell}{n}\binom{n}{n_0,n_1,n_2,\dots}
        \frac{1}{\ell!}\binom{k \ell}{k,k,\dots,k} \prod_{i\geq 1}\left(\frac{1}{(i+1)!}\binom{(i+1)k}{k,k,\dots,k,k}\right)^{n_i}.
    \end{align*}
    The first part contains all summands where the highest degree in the tree is higher than $n-n^{1/3}$ and the second part accounts for all cases where the highest degree is at most $n-n^{1/3}$.\\
    
    \textbf{Part 1: The highest degree is at least ($n-n^{1/3}$).} First of all, we already analyzed the summands that correspond to star-like trees in Proposition~\ref{prop:star}.

    So, for fixed $c>1$, we the highest degree in the tree is $n-c$, we are left with analyzing the sum over all $(n_0,n_1,\dots)\in T_n(\ell)$ with $M(\ell,n_1,\dots) = n-c$ and $\max\{\ell-2,\sum_{i=2}^{n-c-2}n_i\} >0$. Further, we use the Rerooting Lemma~\ref{lem:reroot} and in particular Corollary~\ref{cor:reroot} to simplify the sum to
    \begin{align*} \label{eq:big_deg}
        &\sum_{\ell=1}^{n} 
        \sum_{\substack{(n_0,n_1,\dots)\in T_n(\ell)\\
        \\
        M(\ell,n_1,\dots) \geq n-\lceil n^{1/3}\rceil\\ \max\{\ell-2,\sum_{i=2}^{n-c-2}n_i\} >0}}
        \frac{\ell}{n}\binom{n}{n_0,n_1,n_2,\dots}
        \frac{1}{\ell!}\binom{k \ell}{k,k,\dots,k} \prod_{i\geq 1}\left(\frac{1}{(i+1)!}\binom{(i+1)k}{k,k,\dots,k,k}\right)^{n_i} \nonumber \\
        &\leq 
        \sum_{\substack{(n_0,n_1,\dots)\in T_n(1)\\
        \\
        M(1,n_1,\dots) \geq n-\lceil n^{1/3}\rceil\\ 
        \sum_{i=2}^{n-c-2}n_i >0}}
        \frac{n-n_0+2}{n}\binom{n}{n_0,n_1,n_2,\dots}
        \prod_{i\geq 1}\left(\frac{1}{(i+1)!}\binom{(i+1)k}{k,k,\dots,k,k}\right)^{n_i}. 
    \end{align*}
    
    In the following, we estimate the summand corresponding to $(n_0,n_1,\dots,n_{n-c-1},0\dots)$ in the above sum with respect to $s(n,n_1-1+\sum_{i\geq 2} n_i)$. 
    So, let $m = \sum_{i\geq 1} n_i = n-n_0+1$ be the number of internal vertices. Further, let $n-c-1 > n-\lceil n^{1/3} \rceil$ be the largest outdegree and $p$ be the smallest outdegree larger than $1$ in the degree sequence $(n_0,n_1,\dots)$. We estimate the corresponding summand by comparing it to the summand with degree sequence $(n'_0,n'_1,\dots)$, where $n'_p = n_p-1$, $n'_{p-1} = n_{p-1}+1$, $n'_{n-c-1}=0$ and $n'_{n-c}=1$. (One can think about it as moving one of the subtrees at a vertex of outdegree $p$ to the vertex with the highest outdegree.)
    We therefore obtain
    \begin{align*}
        &\frac{\frac{n-n_0+2}{n}\binom{n}{n_0,n_1,n_2,\dots}
        \prod_{i\geq 1}\left(\frac{1}{(i+1)!}\binom{(i+1)k}{k,k,\dots,k,k}\right)^{n_i}}{\frac{n-n'_0+2}{n}\binom{n}{n'_0,n'_1,n'_2,\dots}
        \prod_{i\geq 1}\left(\frac{1}{(i+1)!}\binom{(i+1)k}{k,k,\dots,k,k}\right)^{n'_i}} = \frac{
        \left(\frac{(k(n-c))!}{(n-c)!}\right)
        \frac{1}{n_{p}}\left(\frac{(k(p+1))!}{(p+1)!}\right)
        }
        {\left(\frac{(k(n-c+1))!}{(n-c+1)!}\right)\frac{1}{n_{p-1}+1}\left(\frac{(kp)!}{p!}\right)}
        \\
        &= \frac{n_{p-1}+1}{n_{p}}\frac{k(p+1)(kp+k-1)\cdots(kp+1)}{k(n-c+1)(kn-kc+k-1)\cdots(kn-kc+1)}\frac{n-c+1}{p+1}
        \\
        &\leq \frac{n_{p-1}+1}{n_{p}}\left(\frac{p+1}{n-c}\right)^{k-1}.
    \end{align*}
    Note that $n_{p-1}=0$ if $p >2$ since it is the smallest outdegree which is larger than $1$. So,
    \begin{equation}\label{eq:vergleich}
        \frac{\frac{n-n_0+2}{n}\binom{n}{n_0,n_1,n_2,\dots}
        \prod_{i\geq 1}\left(\frac{1}{(i+1)!}\binom{(i+1)k}{k,k,\dots,k,k}\right)^{n_i}}{\frac{n-n'_0+2}{n}\binom{n}{n'_0,n'_1,n'_2,\dots}
        \prod_{i\geq 1}\left(\frac{1}{(i+1)!}\binom{(i+1)k}{k,k,\dots,k,k}\right)^{n'_i}} 
        \leq 
        \begin{cases}
            \left(\frac{c+1}{n-c}\right)^{k-1} &\mbox{if } p> 2\\
            (n_1+1)\left(\frac{3}{n-c}\right)^{k-1} &\mbox{if } p=2
        \end{cases}
    \end{equation}
    Thus, if we apply this estimate recursively until we compare the tree to a starlike tree degree sequence, where the maximum degree is $n-c':=n-c+\sum_{i=2}^{n-c-2}(i-1)n_i$, we obtain
    \begin{align*}
        &\frac{\frac{n-n_0+2}{n}\binom{n}{n_0,n_1,n_2,\dots}
        \prod_{i\geq 1}\left(\frac{1}{(i+1)!}\binom{(i+1)k}{k,k,\dots,k,k}\right)^{n_i}}{\frac{n-n_0+2}{n}\binom{n}{n_0,n_1+m,1}
        \left(\frac{1}{2}\binom{2k}{k,k,\dots,k,k}\right)^{n_1+m}\frac{1}{(n-c'+1)}\binom{(n-c'+1)k}{k,k,\dots,k,k}}
        \\
        &\leq \frac{(m-1)!}{n_1!}\left(\left(\frac{3}{n-c}\right)^{m-1-n_1}\left(\frac{c+1}{n-c}\right)^{\sum_{i=2}^{n-c-2}(i-2)n_i}\right)^{k-1}
    \end{align*}

    We know from the proof of Proposition~\ref{prop:star} that
    \begin{align*}
        &\frac{1}{n}\binom{n}{n_0,n_1+m,1}
        \left(\frac{1}{2}\binom{2k}{k,\dots,k}\right)^{m-1}\frac{1}{(n-c'+1)}\binom{(n-c'+1)k}{k,\dots,k} \leq \frac{1}{2}s_k(n,m-1)
    \end{align*}
    Since $\sum_{i=2}^{n-c-2}(i-2)n_i =c+n_1-2(m-1)$, we proceed with
    \begin{align*}
        &\frac{1}{n}\binom{n}{n_0,n_1,n_2,\dots}
        \prod_{i\geq 1}\left(\frac{1}{(i+1)!}\binom{(i+1)k}{k,k,\dots,k,k}\right)^{n_i}
        \\
        &\leq \frac{(m-1)!}{n_1!}\left(\left(\frac{3}{n-c}\right)^{m-1-n_1}\left(\frac{c+1}{n-c}\right)^{c+n_1-2(m-1)}\right)^{k-1} 
        \frac{1}{2}\frac{s_k(n,m-1)}{2}
        \\
        &\leq \frac{(m-1)!}{n_1!}\left(\left(\frac{3}{c+1}\right)^{m-1-n_1}\left(\frac{c+1}{n-c}\right)^{c-(m-1)}\right)^{k-1} 
        \frac{1}{2}\frac{s_k(n,m-1)}{2}
        \\
        &\leq 
        \left(\frac{m-1}{c+1}\right)^{m-1-n_1}3^{m-1-n_1}\left(\frac{c+1}{n-c}\right)^{(k-1)(c-m+1)}\frac{s_k(n,m-1)}{2}\\
        &\leq 3^{m-1}\left(\frac{c+1}{n-c}\right)^{(k-1)(c-m+1)}\frac{s_k(n,m-1)}{2},
    \end{align*}
    where we used that $m-1 < c+1$ due to the fact that $m-1 = \sum_{i\geq 1}n_i-1 < \sum_{i= 1}^{n-c-2} in_i=c$ and $c\geq 2$ since there is at least one vertex with outdegree at least $2$, smaller than $n-c-1$. Further, by Proposition~\ref{prop:star}, we have
    \begin{align*}
        &
        3^{m-1}\left(\frac{c+1}{n-c}\right)^{(k-1)(c-m+1)}\frac{s_k(n,m-1)}{2}\\
        &\leq \left(\frac{c+1}{n-c}\right)^{(k-1)(c-m+1)}\frac{1}{(m-1)!}\left(\frac{3(2k)!}{2k^kk!}\right)^{m-1}\frac{s_k(n,0)}{2n^{(k-2)(m-1)}}.
    \end{align*}
    
    The number of summands that produce this estimate is the number of partitions of $c$ into $m-1$ parts. Therefore, we multiply the bound by $e^{\pi\sqrt{\frac{2(c-m+1)}{3}}}$ which bounds the number of such partitions \cite{par_num} and obtain,
    \begin{align*}
        &\sum_{\substack{(n_0,n_1,\dots)\in T_n(1)\\
        \\
        M(1,n_1,\dots) \geq n-\lceil n^{1/3}\rceil\\ 
        \max\{\ell-2,\sum_{i=2}^{n-c-2}n_i\} >0}}
        \frac{n-n_0+2}{n}\binom{n}{n_0,n_1,n_2,\dots}
        \prod_{i\geq 1}\left(\frac{1}{(i+1)!}\binom{(i+1)k}{k,k,\dots,k,k}\right)^{n_i}
        \\
        &\leq \sum_{c=2}^{\lceil n^{1/3}\rceil }\sum_{m=2}^{c}  
        e^{\pi\sqrt{\frac{2(c-m+1)}{3}}}\frac{m+1}{2(m-1)!}\left(\frac{3^{k-1}(2k)!}{2k^kk!}\right)^{m-1}\left(\frac{c+1}{n-c}\right)^{(k-1)(c-m+1)}\frac{s_k(n,0)}{n^{(k-2)(m-1)}}\\
        &\leq \sum_{c=2}^{\lceil n^{1/3}\rceil }\sum_{m=2}^{c}  
        \frac{6^{(k-1)(m-1)}}{(m-2)!}\left(\frac{e^{\pi\sqrt{\frac{2}{3}}}(c+1)}{n-c}\right)^{(k-1)(c-m+1)}\frac{s_k(n,0)}{n^{(k-2)(m-1)}}\\
        &\leq \sum_{c=2}^{\lceil n^{1/3}\rceil }\left(\frac{14n^{1/3}}{n-n^{1/3}}\right)^{k-1}\frac{s_k(n,0)}{n^{k-2}}\sum_{m=0}^{\infty}  
        \frac{1}{m!}6^{(k-1)(m+1)}\\
        &\leq \sum_{c=2}^{\lceil n^{1/3}\rceil }c(k)\left(\frac{14n^{1/3}}{n-n^{1/3}}\right)^{k-1}\frac{s_k(n,0)}{n^{k-2}},
    \end{align*}
    where the final constant depending on $k$ is $c(k) = 6^{k-1}\exp(6^{k-1})$. In total, we thus obtain
    \begin{align*}
        \sum_{\ell=1}^{n} 
        \sum_{\substack{(n_0,n_1,\dots)\in T_n(\ell)\\
        \\
        M(\ell,n_1,\dots) \geq n-\lceil n^{1/3}\rceil}}
        &\frac{\ell}{n}\binom{n}{n_0,n_1,n_2,\dots}
        \frac{1}{\ell!}\binom{k \ell}{k,k,\dots,k} \prod_{i\geq 1}\left(\frac{1}{(i+1)!}\binom{(i+1)k}{k,k,\dots,k,k}\right)^{n_i} \\
        &\leq \sum_{m=0}^{n^{1/3}}s_k(n,m)+c(k)\left(\frac{14n^{2/3}}{n-n^{1/3}}\right)^{k-1}\frac{s_k(n,0)}{n^{k-2}}\\
        &= \sum_{m=0}^{n^{1/3}}s_k(n,m) +o\left(s_k(n,0)\right).
    \end{align*}
    
    \textbf{Part 2: The highest degree is smaller than ($n-n^{1/3}$).}
    We are left with the case where $c>\lceil n^{1/3} \rceil$, that is, the largest degree in the tree is at most $\alpha := n-\lceil n^{1/3}\rceil-1$. 
    That means we only need to consider the subsum
    \begin{align*}
        S_n(\alpha) := \sum_{\ell =1}^\alpha &\sum_{\substack{\sum n_i = n\\ \sum in_i = n-\ell}} \frac{\ell}{n}\binom{n}{n_0,n_1,n_2,\dots}
        \frac{1}{\ell!}\binom{k \ell}{k,\dots,k} \prod_{i= 1}^{\alpha-1}\left(\frac{1}{(i+1)!}\binom{k(i+1)}{k,\dots,k}\right)^{n_i}
    \end{align*}
    which conveniently can be rewritten as a sum over coefficients of a bivariate generating function,
    \begin{align*}
        S_n(\alpha) &= \sum_{\ell =1}^\alpha \frac{\ell}{n}\frac{n!}{\ell!}\binom{k\ell}{k,\dots,k} [x^ny^{n-\ell}]\prod_{i=1}^{\alpha}\exp\left(\frac{1}{i!}\binom{ki}{k,\dots,k}xy^{i-1}\right)\\
        &= \sum_{\ell =1}^\alpha \frac{\ell}{n}\frac{n!}{\ell!}\binom{k\ell}{k,\dots,k} [x^ny^{n-\ell}] F(x,y)
    \end{align*}
    where
    \[
        F(x,y) := \prod_{i=1}^{\alpha}\exp\left(\frac{1}{i!}\binom{ki}{k,\dots,k}xy^{i-1}\right) = \exp\left(\sum_{i=1}^{\alpha}\frac{1}{i!}\binom{ki}{k,\dots,k}xy^{i-1}\right) = e^{xf(y)},
    \]
    and
    \[
        f(y) = \sum_{i=1}^{\alpha}\frac{1}{i!}\binom{ki}{k,\dots,k} y^{i-1}.
    \]
    Clearly, we can roughly estimate the coefficient of $x^ny^{n-\ell}$ in $F(x,y)$ by the full function
    \[
        [x^ny^{n-\ell}]F(x,y) \leq \frac{e^{xf_n(y)}}{x^ny^{n-\ell}}
    \]
    evaluated at arbitrary positive $x,y >0$. We may therefore choose a classic saddle point bound by
    \[
        y_0=\frac{k!e^{k-1}}{k^kn^{k-1}}, \quad x_0=n.
    \]
    such that 
    \[
        x_0^{-n}y_0^{\ell-n} = \left(\frac{k!e^{k-1}}{k^kn^{k-1}}\right)^\ell\left(\frac{k^kn^{k-2}}{k!e^{k-1}}\right)^n.
    \]
    Next, we estimate the coefficients of $f(y)$ by the upper bound
    \[
        \frac{1}{i!}\binom{ki}{k,\dots,k} \leq \sqrt{k}\left(\frac{k^ki^{k-1}}{k!e^{k-1}}\right)^ie^{\frac{1}{24i}-\frac{1}{12i+1}} <  \sqrt{k}\left(\frac{k^ki^{k-1}}{k!e^{k-1}}\right)^i
    \]
    derived from applying Stirling's approximation. The evaluation of $f(y)$ at $y_0$ is therefore bounded by
    \begin{align*}
        f(y_0) &\leq 1+\sum_{i=2}^{\alpha}\sqrt{k}\left(\frac{k^ki^{k-1}}{k!e^{k-1}}\right)^i\left(\frac{k!e^{k-1}}{k^kn^{k-1}}\right)^{i-1} \\
        &\leq 1+\sqrt{k}\left(\frac{k^k}{k!e^{k-1}}\right)\left(\frac{4}{n}\right)^{k-1}+\sqrt{k}\left(\frac{k^kn^{k-1}}{k!e^{k-1}}\right)\sum_{i=3}^{\alpha}\left(\frac{i}{n}\right)^{(k-1)i}.
    \end{align*}
    Since $i \mapsto \left(\frac{i}{n}\right)^i$ has a unique minimum at $i=\frac{n}{e}$ and admits its maximum at the tails of the sum, we may estimate the summands by the maximum of the values for $i=3$ and $i=n-n^{1/3}$. That is,
    \begin{align*}
        f(y_0)
        &\leq 1+\frac{e}{\sqrt{2\pi}}\left(\frac{4}{n}\right)^{k-1}+\frac{en^{k-1}}{\sqrt{2\pi} } (n-n^{1/3})\left(\max\left\{\frac{27}{n^3}, \left(1+\frac{n^{1/3}}{n-n^{1/3}}\right)^{-(n-n^{1/3})}\right\}\right)^{k-1}\\
        &\leq 1+\frac{e}{\sqrt{2\pi}}\left(\frac{4}{n}\right)^{k-1} + \frac{e}{\sqrt{2\pi}}\left(\frac{27}{n}\right)^{k-1},
    \end{align*} 
    where we also used the fact that $k! \geq \sqrt{2k\pi}\left(\frac{k}{e}\right)^k$.
    Hence, for large $n$,
    \[
        x_0f_n(y_0) \leq n+O\left(\frac{1}{n^{k-2}}\right)
    \]
    and consequently, 
    \[
        [x^ny^{n-\ell}]F(x,y) \leq e^{n+O\left(\frac{1}{n^{k-2}}\right)}\left(\frac{k!e^{k-1}}{k^kn^{k-1}}\right)^\ell\left(\frac{k^kn^{k-2}}{k!e^{k-1}}\right)^n = e^{O(1)}\left(\frac{k!e^{k-1}}{k^kn^{k-1}}\right)^\ell\left(\frac{k^kn^{k-2}}{k!e^{k-2}}\right)^n.
    \]
    If we go back to the subsum $S_n(\alpha)$, we may therefore estimate it by
    \begin{align*}
        S_n(\alpha) &\leq e^{O(1)}\left(\frac{k^kn^{k-2}}{k!e^{k-2}}\right)^n\sum_{\ell =1}^{\alpha} \frac{\ell}{n}\frac{n!}{\ell!}\binom{k\ell}{k,\dots,k} \left(\frac{k!e^{k-1}}{k^kn^{k-1}}\right)^\ell\\
        &\leq e^{O(1)}\left(\frac{k^kn^{k-2}}{k!e^{k-2}}\right)^nn!\sum_{\ell =1}^{\alpha} \sqrt{k}\left(\frac{k^k\ell^{k-1}}{k!e^{k-1}}\right)^\ell\left(\frac{k!e^{k-1}}{k^kn^{k-1}}\right)^\ell
        \\
        &\leq e^{O(1)}\left(\frac{k^kn^{k-2}}{k!e^{k-2}}\right)^nn!\sum_{\ell =1}^{\alpha} \left(\frac{\ell}{n}\right)^
        {\ell(k-1)}\\
        &\leq \frac{e^{O(1)}}{n^{k-1}}\left(\frac{k^k}{k!}\right)^n \left(\frac{n^{k-2}}{e^{k-2}}\right)^n n! \\
        &\leq \frac{e^{O(1)}}{n^{k-1}}\left(\frac{k^k}{k!}\right)^n (n!)^{k-1}\\
        &= o\left(s_k(n,0)\right).
    \end{align*}
    We can finally conclude that
    \begin{align*}
        c_n^{(k)} 
        &= \sum_{c=0}^{n^{1/3}}s_k(n,c) + o\left(s_k(n,0)\right) \sim \begin{cases}
            2\sqrt{\frac{e^3}{\pi n}}2^nn! &\mbox{if } k=2\\
            2 \sqrt{\frac{k}{(2\pi n)^{k-1}}}\left(\frac{k^k}{k!}\right)^n(n!)^{k-1} &\mbox{if } k>2
        \end{cases}
    \end{align*}
    as $n\rightarrow \infty$.
\end{proof}

\section*{Acknowledgments}
The author was informed by Paul Gunnells that Mario DeFranco independently found the formula for the hypergraph Catalan numbers and computed their asymptotic expansion during his employment at UMass, but he left academia before formally writing up his results. She would also like to thank her supervisor Michael Drmota for the helpful discussions.


\bibliography{walks.bib}{}

@article{Bushygraphs,
author = {Enriquez, Nathana\"el and M\'enard, Laurent},
year = {2013},
month = {10},
pages = {},
title = {Spectra of large diluted but bushy random graphs},
volume = {49},
journal = {Random Structures \& Algorithms},
doi = {10.1002/rsa.20618}
}

@article{inna_z,
author = {Zakharevich, Inna},
year = {2006},
pages = {403–414},
volume = {268},
title = {A Generalization of Wigner’s Law},
journal={Communications in Mathematical Physics},
doi = {10.1007/s00220-006-0074-5}
}

@article{physics,
author = {Bauer, Michel and Golinelli, Olivier},
year = {2001},
pages = {301–337},
volume = {103},
title = {Random Incidence Matrices: Moments of the Spectral Density},
journal = {Journal of Statistical Physics},
doi = {10.1023/A:1004879905284}
}

@article{spec1,
author = {Bordenave,Charles and Lelarge, Marc and Salez, Justin},
year = {2011},
title = {The rank of diluted random graphs},
journal = {The Annals of Probability 39(3)},
pages = {1097-1121}
}

@article{spec2,
author = {Salez, Justin},
year = {2015},
title = {Every totally real algebraic integer is a tree eigenvalue},
journal = {Journal of Combinatorial Theory, Series B},
pages = {249-256}
}

@article{spec3,
author = {Bordenave, Charles and Sen, A. and Vir\'ag, B.},
year = {2017},
title = {Mean quantum percolation},
journal = {Journal of the European Mathematical Society 19},
pages = {3679–3707}
}

@article{spec4,
author = {Salez, Justin},
year = {2020},
title = {Spectral atoms of unimodular random trees},
journal = {Journal of the European Mathematical Society 22},
pages = {345–363}
}

@article{gunn,
author = {Gunnels, Paul},
year = {2021},
title = {Generalized Catalan numbers from hypergraphs},
journal = {Electronic Journal of Combinatorics},
volume = {28},
issue= {1}
}

@InProceedings{pana_ha,
  author =	{Hainzl, Eva-Maria and de Panafieu, \'{E}lie},
  title =	{{Tree Walks and the Spectrum of Random Graphs}},
  booktitle =	{35th International Conference on Probabilistic, Combinatorial and Asymptotic Methods for the Analysis of Algorithms (AofA 2024)},
  pages =	{11:1--11:15},
  series =	{Leibniz International Proceedings in Informatics (LIPIcs)},
  ISBN =	{978-3-95977-329-4},
  ISSN =	{1868-8969},
  year =	{2024},
  volume =	{302},
  editor =	{Mailler, C\'{e}cile and Wild, Sebastian},
  publisher =	{Schloss Dagstuhl -- Leibniz-Zentrum f{\"u}r Informatik},
  address =	{Dagstuhl, Germany},
  URL =		{https://drops.dagstuhl.de/entities/document/10.4230/LIPIcs.AofA.2024.11},
  URN =		{urn:nbn:de:0030-drops-204466},
  doi =		{10.4230/LIPIcs.AofA.2024.11}
}

@inbook{gess,
author = {Gessel, Ira M. and Stanley, Richard P.},
title = {Algebraic enumeration},
year = {1996},
isbn = {0262071711},
publisher = {MIT Press},
address = {Cambridge, MA, USA},
booktitle = {Handbook of Combinatorics (Vol. 2)},
pages = {1021–1061},
numpages = {41}
}

@article{par_num,
title = {On number of partitions of an integer into a fixed number of positive integers},
journal = {Journal of Number Theory},
volume = {159},
pages = {355-369},
year = {2016},
issn = {0022-314X},
doi = {https://doi.org/10.1016/j.jnt.2015.06.023},
url = {https://www.sciencedirect.com/science/article/pii/S0022314X1500236X},
author = {A. Yavuz Oruç}
}

@article{matrixmod,
author = {Gunnells, Paul},
year = {2024},
month = {07},
pages = {149-178},
title = {Hypergraph matrix models and generating functions},
volume = {13},
journal = {Combinatorics and Number Theory},
doi = {10.2140/cnt.2024.13.149}
}

@book{Catalan,
title = {Catalan numbers},
subtitel = {},
publisher = {Cambridge University Press},
ISBN = {978-1-107-07509-2},
year = {2015},
doi = {},
author = {R. {Stanley}}
}
\bibliographystyle{plain}
\end{document}